\documentclass{article}

\usepackage{polski}
\usepackage[english]{babel}
\usepackage[utf8]{inputenc}
\usepackage{amsthm}
\usepackage{amsmath}
\usepackage{amsfonts}
\usepackage{amssymb}
\usepackage{mathtools}
\usepackage{mathrsfs}
\usepackage[a4paper, left=3cm, right=3cm, top=3cm, bottom=3cm]{geometry}
\usepackage{url}
\usepackage[all]{xy}

\newcommand{\wt}[1]{{\widetilde{#1}}}
\newcommand{\injr}[1]{{\frac{\pi}{#1\sqrt{K}}}}
\providecommand{\dim}{\mathop{\rm dim}\nolimits}
\providecommand{\deg}{\mathop{\rm deg}\nolimits}
\providecommand{\vol}{\mathop{\rm vol}\nolimits}
\providecommand{\const}{\mathop{\rm const}\nolimits}
\providecommand{\Ricci}{\mathop{\rm Ricci}\nolimits}
\providecommand{\sec}{\mathop{\rm sec}\nolimits}
\providecommand{\diam}{\mathop{\rm diam}\nolimits}
\providecommand{\im}{\mathop{\rm im}\nolimits}
\providecommand{\exp}{\mathop{\rm exp}\nolimits}
\providecommand{\str}{\mathop{\rm str}\nolimits}
\providecommand{\supp}{\mathop{\rm supp}\nolimits}
\providecommand{\Lip}{\mathop{\rm Lip}\nolimits}
\providecommand{\dvol}{\mathop{\rm dvol}\nolimits}
\providecommand{\smear}{\mathop{\rm smear}\nolimits}

\newtheorem{thm}{Theorem}[section]
\newtheorem{prop}[thm]{Proposition}
\newtheorem{lemma}[thm]{Lemma}
\newtheorem{cor}[thm]{Corollary}

\theoremstyle{definition}

\newtheorem{remark}[thm]{Remark}
\newtheorem{defi}[thm]{Definition}

\title{Piecewise straightening and Lipschitz simplicial volume}
\author{Karol Strzałkowski}

\begin{document}
\maketitle

\begin{abstract}
We study the Lipschitz simplicial volume, which is a metric version of the simplicial volume. We introduce the piecewise straightening procedure for singular chains, which allows us to generalize the proportionality principle and the product inequality to the case of complete Riemannian manifolds of finite volume with sectional curvature bounded from above. We obtain also yet another proof of the proportionality principle in the compact case by a direct approximation of the smearing map.
\end{abstract}

\section{Introduction}
The simplicial volume is a homotopy invariant of manifolds defined for a closed manifold $M$ as
$$
\|M\| :=inf\{|c|_1\::\: \text{$c$ is a fundamental cycle with $\mathbb{R}$ coefficients}\},
$$
where $|\cdot|_1$ is an $\ell^1$-norm on $C_*(M,\mathbb{R})$ (which we will denote for simplicity as $C_*(M)$) with respect to the basis consisting of singular simplices. Although the definition is relatively straightforward, it has many applications. Most of them are mentioned in the work of Gromov \cite{G}, one of the most important is the use to the degree theorems. In general, by the degree theorem we understand a bound on the degree of a continuous map $f:M\rightarrow N$ between two $n$-dimensional Riemannian manifolds
$$
\deg(f)\leq \const_n\frac{\vol(M)}{\vol(N)}.
$$
Such a theorem may obviously require additional assumptions. The reason why the simplicial volume is suitable for establishing such theorems is its functoriality, i.e. if $f:M\rightarrow N$ is a map between two closed manifolds then
$$
\|N\|\leq \deg(f)\cdot \|M\|.
$$
One obtains easily that if $\|N\|\neq0$, then
$$
\deg(f)\leq \frac{\|M\|}{\|N\|}.
$$
Under some curvature assumptions, Gromov proved in \cite{G} that for a given Riemannian manifold $M$ we have $\|M\|\leq \const_n\cdot \vol(M)$ and $\|M\|\geq \const_n\cdot \vol(M)$, which imply the degree theorem if the curvature assumptions are satisfied.

In most cases simplicial volume is very difficult to compute exactly. However, it has a few properties which can be used to approximate it or at least decide if it is zero or not. Two of them which we are interested in are the product inequality and the proportionality principle.

\begin{thm}[\cite{G}]\label{PI0}
  Let $M$ and $N$ be two compact manifolds. Then the following inequality holds
 $$
 \|M\|\cdot\|N\| \leq \|M\times N\|.
 $$
\end{thm}

\begin{thm}[\cite{G}, \cite{Str}]\label{PP0}
 Let $M$ and $N$ be two compact Riemannian manifolds. Assume also that their universal covers are isometric. Then
 $$
 \frac{\|M\|}{\vol(M)} = \frac{\|N\|}{\vol(N)}.
 $$
\end{thm}

A natural question to ask is if these properties generalise somehow to the non-compact case. In order to have a fundamental class, one needs to consider $\ell^1$ norm on locally finite singular chains instead of just (finite) singular chains. In this case simplicial volume obviously does not have to be finite. Unfortunately, neither of the above properties holds in such generality. The product inequality does not hold because of another result of Gromov from \cite{G} that the simplicial volume of a product of at least 3 open manifolds is $0$, while there are examples of products of two such manifolds with nonzero simplicial volume \cite{LS}. The proportionality principle fails because of a similar reason. Take a product of three non-compact, locally symmetric space of finite volume. Its simplicial volume vanishes, but on the other hand there always exists a compact locally symmetric space with isometric universal cover \cite{B} and the simplicial volume of closed locally symmetric spaces of non-compact type is known to be nonzero \cite{LafS}.

The solution to these problems, also proposed by Gromov in \cite{G}, is to consider a geometric variant of simplicial volume by taking only Lipschitz chains. This way one obtains the Lipschitz simplicial volume
$$
\|M\|_{\Lip} :=inf\{|c|_1\::\: \text{$c\in C^{lf}_*(M)$ is a fundamental cycle with $\mathbb{R}$ coefficients, $\Lip(c)<\infty$}\}.
$$
In the case of closed manifolds the classical and the Lipschitz simplicial volumes coincide. L\"oh and Sauer studied the above invariant in \cite{LS} and proved that it may be a proper generalisation of the simplicial volume to the case of complete Riemannian manifolds of finite volume, not necessarily compact. In particular, in the presence of non-positive curvature they proved the proportionality principle and the product inequality. The main result of this article is a generalisation of their proofs to the case of manifolds with curvature bounded from above.

\begin{thm}[Product inequality]\label{PI1}
  Let $M$ and $N$ be two complete, Riemannian manifolds with sectional curvatures bounded from above. Then the following inequality holds
 $$
 \|M\|_{\Lip}\cdot\|N\|_{\Lip} \leq \|M\times N\|_{\Lip}.
 $$
\end{thm}

\begin{thm}[Proportionality principle]\label{PP1}
 Let $M$ and $N$ be two complete Riemannian manifolds of finite volume with sectional curvatures bounded from above. Assume also that their universal covers are isometric. Then
 $$
 \frac{\|M\|_{\Lip}}{\vol(M)} = \frac{\|N\|_{\Lip}}{\vol(N)}.
 $$
\end{thm}

In the work of L\"oh and Sauer, non-positive curvature assumption is needed to introduce the procedure of straightening the simplices. Namely, given a singular chain, one can homotope it to the chain consisting of straight simplices by using the fact that in simply connected, non-positively curved Riemannian manifolds geodesics are unique. To generalize this straightening technique to the case of manifolds with curvature bounded from above we construct 'exponential neighbourhoods' of points of a given manifold, where the straightening can be applied after some iterated barycentric subdivision of simplices. These 'neighbourhoods' were introduced by Gromov in \cite[4.3(B)]{G}, however, we give much more details.

Since for closed manifolds sectional curvature is always bounded and the Lipschitz simplicial volume equals the classical one, we obtain yet another proof of Theorem \ref{PP0}. We follow Thurston's approach from \cite{T} (used in \cite{Str}), however, we obtain the proof without any use of bounded cohomology by approximating directly the smearing map.

The proportionality principle provides direct connection between Lipschitz simplicial volume and volume, therefore one obtains immediately

\begin{cor}
 Let $f:M\rightarrow N$ be a proper Lipschitz map between two complete Riemannian manifolds of finite volume with sectional curvatures bounded from above, which in addition have isometric universal covers. Assume moreover that $\|N\|_{\Lip}\neq 0$. Then
 $$
 \deg(f)\leq \frac{\vol(M)}{\vol(N)}.
 $$
\end{cor}

L\"oh and Sauer combined this fact for non-positively curved manifolds with the facts that Lipschitz simplicial volume is strictly positive for locally symmetric spaces of non-compact type of finite volume \cite{B,LafS,LS}, that there are only finitely many symmetric spaces (with the standard metric) in each dimension and that $\|N\|\leq C_n \vol(N)$ if $\Ricci(N)\geq -(n-1)$ and $\sec(N)\leq 1$ \cite{G,LS} to prove the following theorem.

\begin{thm}[Degree theorem, \cite{CF,LafS,LS}]
 For every $n\in\mathbb{N}$ there is a constant $C_n>0$ with the following property: Let $M$ be an $n$-dimensional locally symmetric space of non-compact type with finite volume. Let $N$ be an $n$-dimensional complete Riemannian manifold of finite volume with $\Ricci(N)\geq -(n-1)$ and $\sec(N)\leq 1$, and let $f:N\rightarrow M$ be a proper Lipschitz map. Then
 $$
  \deg(f)\leq C_n\cdot\frac{\vol(N)}{\vol(M)}.
 $$
\end{thm}

Possible generalisation of the above theorem depends on the further results on non-vanishing of the Lipschitz simplicial volume. At the moment most results in this direction are based on the proportionality principle indicated above, non-vanishing of the simplicial volume for negatively curved spaces \cite{T}, locally symmetric spaces of non-compact type \cite{LafS} and their products and connected sums \cite{G}.

\subsection*{Notation}
To clarify the notation, we will denote by $B_M(x,r)$ an open ball in a space $M$ centred at $x$ with radius $r$, and more generally by $B_M(X,r)$ an open $r$-neighbourhood of a set $X\subset M$. We consider all Riemannian manifolds as metric spaces with metric induced by Riemannian structure. In particular, if they are complete they are geodesic as metric spaces i.e. the distance between two points equals the length of shortest path joining them. We will also identify a $k$-dimensional simplex $\Delta^k$ with a set $\{(x_0,..,x_k)\in\mathbb{R}_{\geq 0}^{k+1}\::\:\sum_{i=0}^kx_i=1\,,\}\subset\mathbb{R}^{k+1}$ with an induced Riemannian structure.

\subsection*{Organization of this work}
In Section~\ref{SecPiece} we define the 'exponential neighbourhoods', recall the basic facts about straight simplices and develop the piecewise straightening procedure for singular chains. In Section~\ref{SecHom} we define piecewise $C^1$ chains and introduce corresponding singular and Milnor-Thurston-type homology theories. Section~\ref{SecApp} is devoted to the proofs of theorems \ref{PP0}, \ref{PI1} and \ref{PP1}.

\subsection*{Acknowledgements}
I am grateful to Piotr Nowak for bringing the simplicial volume and \cite{LS} to my attention. I would also like to thank Clara L\"{o}h for discussions about \cite{LS}, Jean-Fran\c{c}ois Lafont for an inspiring conversation and Federico Franceschini for very useful comments on the previous versions of this paper.

\section{Piecewise straightening procedure}\label{SecPiece}
The straightening procedure on non-positively curved manifolds is well known and applied successfully to many problems. Roughly speaking, given a complete, simply connected Riemannian manifold $M$ with non-positive curvature and a singular simplex $\sigma:\Delta^k\rightarrow M$ with vertices $x_0,...,x_k$ the straightening of this simplex is the geodesic simplex $[x_0,...,x_k]$, which is defined inductively to be a geodesic join of $x_k$ with geodesic simplex $[x_0,...,x_{k-1}]$. Because geodesics on $M$ joining points are unique, there exists a (unique) geodesic homotopy between $\sigma$ and $[x_0,...,x_k]$ which is defined as the geodesic join $[\sigma,[x_0,...,x_k]]$. We can apply the same procedure to the singular simplex $\sigma$ on non necessarily simply-connected Riemannian manifold $M$ with non-positive curvature by taking its lift to the universal cover $\wt{\sigma}:\Delta^k\rightarrow\wt{M}$, applying straightening there and pushing down the homotopy. It can be shown that it does not depend on the choice of the lift, therefore it can be extended to the straightening on singular chains and induces an isomorphism on homology. The same applies to locally finite Lipschitz chains and homology. The straightening procedure has also the advantage that it does not increase $l_1$-norm on chains, therefore the above isomorphism turns out to be isometric and the simplicial volume can be computed by considering only straight chains. This fact, together with a careful control of the set of vertices, is a key to prove e.g. proportionality principle for Lipschitz simplicial volume and inequalities for products of manifolds, assuming all the manifolds have non-positive curvature.

The fact which obviously fails if we consider a Riemannian manifold with $\sec(M)<K<\infty$ is the existence of unique geodesics on simply connected manifolds. They do exist locally, but unfortunately not uniformly, even if we pass to the universal covering. Therefore the crucial problem in defining piecewise straightening procedure on $M$ is the choice of a suitable space in which we have such uniform local uniqueness of geodesics. If such a space is provided, one can define piecewise straightening by subdividing barycentrically given singular chain, straighten every small simplex and glue straightened simplices back.

In Section~\ref{EN} for every point of $M$ we construct an 'exponential neighbourhood' of it which is a space admitting a local isometry to $M$ for which there exists a uniform lower bound (depending on $K$) of the injectivity radius of points in some (uniform) neighbourhood of the origin. This system of spaces and local isometries on $M$ admits also transition maps (at least locally) which allow one to apply some local constructions independently of the choice of point for which we consider its exponential neighbourhood. The construction is sketched in \cite[4.3(B)]{G}, however, we provide more detailed approach. In Section~\ref{SSAH} we recall basic notions concerning geodesic simplices and joins and prove that under some curvature and diameter conditions a geodesic join of Lipschitz maps is also Lipschitz. Finally, in Section~\ref{PSP} we define the piecewise straightening procedure for locally finite Lipschitz chains.

\subsection{Exponential neighbourhoods}\label{EN}
Let $M$ be a connected, complete $n$-dimensional Riemannian manifold with $\sec(M)<K$, $K>0$.

\begin{defi}
Let $x\in M$ and let $r\leq\injr{}$. Consider an open ball $B_{T_xM}(0,r)$ in the tangent space $T_xM$. Then the exponential map $\exp_x:B_{T_xM}(0,r)\rightarrow M$ is an immersion by Rauch-Berger comparison theorem \cite[Theorems 1.28, 1.29]{CE}. We endow $B_{T_xM}(0,r)$ with a Riemannian metric induced from $M$ by $\exp_x$ and obtain a space $V_x(r)$ which we call \emph{an $r-$exponential neighbourhood} of $x$ with distinguished point $\bar{x}\in V_x(r)$, which corresponds to $0$ in $B_{T_xM}(0,r)$ and the canonical local isometry $p_x:V_x(r)\rightarrow M$ such that $p_x(\bar{x})=x$.

If $r=\injr{}$, we will denote this space for short as $V_x$. 
\end{defi}

Spaces $V_x$ are not complete, however, the closures of open balls $B_{V_x}(\bar{x},r)$ for any $r<\injr{}$ are complete as metric spaces and for $y\in V_x(r)$ the map $\exp_y:T_yV_x\rightarrow V_x$ is defined for vectors of length less than $\injr{}-r$. As we will see next, these spaces have all desired properties described in the introduction of this section. First of all we check that there exists a uniform lower bound on the injectivity radii of points around the (uniform) origins of $V_x$. 

\begin{prop}\label{inj_rad_prop}
 Let $x\in M$ and let $y\in V_x(\injr{4})$. Then the injectivity radius of $y$ in $V_x$ is at least $\injr{4}$.
\end{prop}

\begin{proof}
If $y\in B_{V_x}(\bar{x},\injr{4})$ then the exponential map $\exp_y:T_yV_x\rightarrow V_x$ is defined for vectors of length less than $\frac{3\pi}{4\sqrt{K}}$. Because of the curvature bound, it is immersion by Rauch-Berger comparison theorem, so we only need to prove that it is injective on $B_{T_yV_x}(0,\injr{4})$.

Denote by $V'_y$ the space $B_{T_yV_x}(0,\injr{2})$ endowed with Riemannian metric induced from $V_x$ (in particular the exponential map $\exp_y:V'_y\rightarrow V_x$ becomes a local isometry) with distinguished point $\bar{y}$ corresponding to $0$ in $T_yV_x$. Let $z_1, z_2\in B_{V'_y}(\bar{y},\injr{4})$ be such that $\exp_y(z_1)=\exp_y(z_2)=z$ and let $\wt{x}\in B_{V'_y}(\bar{y},\injr{4})$ be some lift of $\bar{x}$, i.e. any point satisfying $\exp_y(\wt{x})=\bar{x}$. Such point exists in $B_{V'_y}(\bar{y}'\injr{4})$ because $d_{V_x}(\bar{x},y)<\injr{4}$, but need not be unique. Since $\wt{x},z_1,z_2\in B_{V'_y}(\bar{y},\injr{4})$ there exists a geodesic $\gamma_1$ in $V'_y$ joining $z_1$ and $\wt{x}$. Similarly, there is a geodesic $\gamma_2$ joining $z_2$ and $\wt{x}$. Because $\exp_y$ is a local isometry on $V'_y$, both $\exp_y(\gamma_1)$ and $\exp_y(\gamma_2)$ are geodesics joining $\bar{x}$ and $z$ inside $V_x$. But by the construction of the exponential map and the space $V_x$, all geodesics joining $\bar{x}$ and any other point inside $V_x$ are unique. In particular $\exp_y(\gamma_1)=\exp_y(\gamma_2)$. We use again the fact that $\exp_y$ is a local isometry around $\wt{x}$ to see that both geodesics $\gamma_1$ and $\gamma_2$ have the same tangent line in $\wt{x}$ and the same direction, hence (without loss of generality) $\gamma_1$ is a subgeodesic of $\gamma_2$. Moreover, because $\exp_y$ does not change the length of geodesics, we have in fact $\gamma_1=\gamma_2$, hence $z_1=z_2$ q.e.d.
\end{proof}

Secondly, we check the existence of 'transition maps' which will allow us to preform local constructions on spaces $V_x$ independently of $x\in M$.

\begin{prop}\label{transition_prop}
Let $x,y\in M$ be such that $d_M(x,y)<\injr{4}$. Let also $y'$ be any lift of $y$ to $V_x(\injr{4})$. Then there exists a locally isometric diffeomorphism $I_{y',x}:V_y(\injr{4})\rightarrow B_{V_x}(y',\injr{4})$ such that we have a commutative diagram
$$
\xymatrix{
V_y(\injr{4}) \ar@{>}[rr]^{I_{y',x}} \ar@{>}[rd]_{p_{y}} && B_{V_x}(y',\injr{4})\ar@{>}[ld]^{p_x} \\
& M &
}
$$
\end{prop}

\begin{proof}
 By Proposition \ref{inj_rad_prop} we know that $\exp_{y'}$ provides a diffeomorphism
 $$
 \exp_{y'}:B_{T_{y'}V_x}(0,\injr{4})\rightarrow B_{V_x}(y',\injr{4})\subset V_x
 $$
 which can be corrected to be a local isometry by changing the Riemannian metric on $B_{T_{y'}V_x}(0,\injr{4})$. Hence it suffices to show that $V_y(\injr{4})$ is isometric to $B_{T_{y'}V_x}(0,\injr{4})$ (with Riemannian metric induced by $\exp_y$). However, both spaces can be identified with the space of geodesics of length less than $\injr{4}$ starting from $y$, with Riemannian metric induced from $M$ by the map mapping geodesic to its endpoint. Checking the commutativity of a diagram is straightforward. 
\end{proof}

Finally we establish the lifting property for spaces $V_x$ with respect to singular simplices with sufficiently small Lipschitz constants. Recall that if $X$ is a metric space and $\gamma:[0,1]\rightarrow X$ then we define the \emph{length of $\gamma$} to be
$$
L(\gamma) := \sup\{\sum_{i=1}^n d_X(\gamma(t_{i-1}),\gamma(t_i))\::\: 0=t_0<t_1<...<t_n=1,\,n\in\mathbb{N}\}
$$
and we say that $X$ is \emph{geodesic} if it is path-connected and for any two points their distance equals the length of the shortest path between them, called \emph{geodesic}. We will use the following simple fact.

\begin{lemma}\label{lemma_geodesic_lip}
 Let $X$ be a geodesic metric space and let $f:X\rightarrow Y$ be a Lipschitz map. Then for every $\varepsilon>0$
 $$
 \Lip(f) = \sup\{\frac{d_Y(f(x),f(x'))}{d_X(x,x')}\::\: x,x'\in X \,,\, 0<d_X(x,x')<\varepsilon\}
 $$
\end{lemma}

\begin{proof}
 The '$\geq$' inequality is obvious, we need to prove the opposite one. Let $\delta>0$ and let $x,x'\in X$ be two points such that
 $$
 d_Y(f(x),f(x'))>(\Lip(f)-\delta)d_X(x,x')
 $$ 
 Let also $\gamma:[0,d_X(x,x')]\rightarrow X$ be the shortest geodesic joining $x$ and $x'$. Subdivide $\gamma$ into nontrivial subgeodesics $\gamma_1,...,\gamma_n$ of length less than $\varepsilon$ and let $x=x_0,x_1,...,x_n=x'$ be their subsequent endpoints. Then we have
 $$
 \sum_{i=1}^n d_Y(f(x_{i-1}),f(x_i)) \geq d_Y(f(x),f(x')) > (\Lip(f)-\delta)d_X(x,x') = (\Lip(f)-\delta)\sum_{i=1}^n d_X(x_{i-1},x_i)
 $$
 hence for some $i\in\{1,...,n\}$ we have the inequality
 $$
 d_Y(f(x_{i-1}),f(x_i))>(\Lip(f)-\delta)d_X(x_{i-1},x_i)
 $$
 and $0<d_X(x_{i-1},x_i)<\varepsilon$. Because $\delta$ was arbitrary, the inequality holds.
\end{proof}

Note that every complete Riemannian manifold is geodesic as a metric space. In particular, a $k$-dimensional simplex $\Delta^k$ is a geodesic space with diameter $\sqrt{2}$.

\begin{prop}
 Let $\sigma:\Delta^k\rightarrow M$ be a Lipschitz singular simplex, let $y\in\Delta^k$ and let $\sigma(y)=x\in M$. Then if $\Lip(\sigma)<\frac{C}{\sqrt{2}} <\frac{\pi}{\sqrt{2K}}$ then there exists a unique Lipschitz lift $\wt{\sigma}:\Delta^k\rightarrow V_x(C)$ of $\sigma$ (i.e. $\sigma = p_x\circ\wt{\sigma}$) such that $\wt{\sigma}(y)=\bar{x}$. This lift satisfies also $\Lip(\wt{\sigma})=\Lip(\sigma)$
 \end{prop}

\begin{proof}
Let $z\in \Delta^k$ and let $I_z:[0,1]\rightarrow \Delta^k$ be a (rescaled) interval connecting $y$ and $z$, that is $I_z(t)=(1-t)y+tz$. Let also $\gamma_z=\sigma\circ I_z$. We claim that we can construct a unique path $\wt{\gamma}_z:[0,1]\rightarrow V_x(C)$ such that $p_x\circ\wt{\gamma}_z= \gamma_z$ and $\wt{\gamma}_z(0)=\bar{x}$. Let
$$
R = \sup\,\{r\in [0,1] \: :\: \text{ there exists a lift $\wt{\gamma}^r_z:[0,r]\rightarrow V_x(C)$ of $\gamma_z|_{[0,r]}$ such that $\wt{\gamma}^r_z(0)=\bar{x}$} \}.
$$
We claim that $R=1$. Note that if we have two lifts $\wt{\gamma}^s_z:[0,s]\rightarrow V_x$ and $\wt{\gamma}^t_z:[0,t]\rightarrow V_x$ for $0\leq s\leq t\leq1$ satisfying the above conditions then they need to agree on $[0,s]$ because the subset of $[0,s]$ where these two lifts agree is nonempty (because $\wt{\gamma}^s_z(0)=\wt{\gamma}^t_z(0)=\bar{x}$), open (because $p_x$ is a local diffeomorphism) and closed (because of the continuity of both lifts). Hence we can consider a (unique) union of such lifts $\wt{\gamma}^s_z$ for $s<R$ to obtain a lift $\wt{\gamma}'^R_z:[0,R)\rightarrow V_x$ of $\gamma_z|_{[0,R)}$ such that $\wt{\gamma}'^R_z(0)=\bar{x}$. To extend it continuously to a lift $\wt{\gamma}^R_z:[0,R]\rightarrow V_x(C)$ we need to check that 
$$
\sup_{t\in[0,R)}d_{V_x}(\bar{x},\wt{\gamma}'^{R}_z(t)) < C
$$
because then the limit $\lim_{t\rightarrow R}\wt{\gamma}'^R_z(t)$ exists in $V_x(C)$. Fix $0<t<R$ and consider a path $\wt{\gamma}_z^t=\wt{\gamma}'^R_z|[0,t]$. Note that because $p_x$ is a local isometry this path has the same length as $\gamma_z|[0,t]$. Using the fact that $\gamma_z=\sigma\circ I_z$ and that $\sigma$ is Lipschitz we have
$$
d_{V_x}(\bar{x},\wt{\gamma}^t_z(t)) = d_{V_x}(\wt{\gamma}^t_z(0),\wt{\gamma}^t_z(t)) \leq L(\wt{\gamma}^t_z) = L(\gamma_z|[0,t]) < (\frac{C}{\sqrt{2}}-\varepsilon) L(I_z) \leq C - \varepsilon
$$
for some sufficiently small $\varepsilon$ depending on $\sigma$, but neither on $z$ nor on $t$. Since $\wt{\gamma}^t_z(t) = \wt{\gamma}'^R_z(t)$ we have $\sup_{t\in[0,R)}d_{V_x}(\bar{x},\wt{\gamma}'^{R}_z(t))\leq C-\varepsilon < C$ so we can extend our lift to $\wt{\gamma}^R_z:[0,R]\rightarrow V_x(C)$. Finally, if $R<1$ we can use again the fact that $p_x$ is a local diffeomorphism (this time in the neighbourhood of $\wt{\gamma}^{R}_z(R)$) and extend $\wt{\gamma}^R_z$ to $\wt{\gamma}^{R'}_z$ for some $R'>R$, which contradicts the definition of $R$.

Because the choice of $\wt{\gamma}_z$ is unique we can define $\wt{\sigma}(z) = \wt{\gamma}_z(1)$. Moreover, we can once again use the fact that $p_x$ is a local diffeomorphism and that $[0,1]$ is compact to conclude that $\wt{\gamma}_z$ depends continuously on $z$ in the compact-open topology, hence $\wt{\sigma}$ as a map $\Delta\rightarrow V_x(C)$ is continuous.

The last claim to verify is the equality $\Lip(\wt{\sigma})=\Lip(\sigma)$. Note that $\Delta^k$ is a geodesic metric space, hence the Lipschitz constants of $\sigma$ and $\wt{\sigma}$ can be computed locally as in Lemma \ref{lemma_geodesic_lip}. But $p_x\circ\wt{\sigma}=\sigma$ and $p_x$ is a local isometry, hence these 'local' Lipschitz constants are the same. 
\end{proof}

By combining the above proposition with Proposition \ref{transition_prop} we obtain very useful corollary.

\begin{cor}\label{simplex_lift_cor}
 Let $\sigma:\Delta^k\rightarrow M$ be a Lipschitz singular simplex such that $\sigma(\Delta^k)\subset B_M(x,\injr{4})$. Then if $\Lip(\sigma)<\frac{C}{\sqrt{2}} <\frac{\pi}{4\sqrt{2K}}$ then there exists a Lipschitz lift $\wt{\sigma}:\Delta^k\rightarrow V_x$ of $\sigma$ (i.e. $p_x\circ\wt{\sigma} = \sigma$) with $\Lip(\wt{\sigma})=\Lip(\sigma)$.
 
 Moreover, if $y\in\Delta^k$ then the lift is unique up to the choice of $\wt{\sigma}(y)$ which can be chosen to be any point $\wt{y}\in V_x(\injr{4})$ such that $p_x\circ\wt{\sigma}(\wt{y})=y$. We have then $\wt{\sigma}(\Delta^k)\subset B_{V_x}(\wt{y},C)$.
\end{cor}

\subsection{Straight simplices and homotopies}\label{SSAH}
As before, we will assume that $M$ is connected complete $n$-dimensional Riemannian manifold with $\sec(M)<K$, $K>0$ and $x\in M$. Let $y,z\in V_x$ be two points such that $y,z\in V_x(\injr{8})$. By Proposition \ref{inj_rad_prop} there exists a unique shortest geodesic joining them (depending continuously on both endpoints) which we denote by $[y,z]$. Following \cite{LS}, we can define the \emph{geodesic join} of two maps $f,g:X\rightarrow V_x$.

\begin{defi}
 Let $f,g:Y\rightarrow V_x$ be two maps such that $(\im(f)\cup \im(g))\subset V_x(\injr{8})$. Then there exists a unique homotopy $[f,g]:Y\times [0,1]\rightarrow V_x$ defined by $(y,t)\mapsto [f(y),g(y)](t)$ called a \emph{geodesic join} of $f$ and $g$.
\end{defi}

We will often use the following lemma.

\begin{lemma}\label{joinlemma}
 Let $f,g:Y\rightarrow V_x$ be two maps such that $\im(f)\subset V_x(R_1)$ and $\im(g)\subset V_x(R_2)$ for $R_1,R_2<\injr{8}$. Then $\im([f,g])\subset V_x(R_1+R_2)$.
\end{lemma}

\begin{proof}
 Suppose there is a point $z = [f,g](y,t)$ such that $d(\bar{x},z)\geq R_1+R_2$. Then
 $$
 d_{V_x}(z,f(y))\geq d_{V_x}(\bar{x},z)-d_{V_x}(\bar{x},f(y))\geq R_2
 $$ and similarly $d_{V_x}(z, g(y))\geq R_1$. Because $z$ is on the unique minimizing geodesic between $f(y)$ and $g(y)$, we have
 $$
 d_{V_x}(f(y),g(y)) = d_V(f(y),z)+d_V(z,g(y)) \geq R_1+R_2.
 $$
 On the other hand
 $$
 d_{V_x}(f(y),g(y)) \leq d_{V_x}(f(y),\bar{x})+d_{V_x}(\bar{x},g(y)) < R_1+R_2.
 $$
 The above contradiction shows that $z\in V_x(R_1+R_2)$.
\end{proof}

We can consequently define geodesic simplices. Recall that as we identified the standard simplex $\Delta^k$ with the subset $\{(z_0,...,z_k)\in\mathbb{R}^{k+1}_{\geq 0}\;:\;\sum_{i=0}^kz_i=1\}$, we can identify $\Delta^{k-1}$ with the subset $\{(z_0,...,z_k)\in\Delta^k\;:\; z_k=0\}$.

\begin{defi}
 The \emph{geodesic simplex} $[x_0,...,x_k]:\Delta^k\rightarrow V_x$ with vertices $x_0,...,x_k\in V_x(\injr{8k})$ is defined inductively by the formulas
 \begin{itemize}
  \item $[x_0](\Delta^0) = \{x_0\}\subset V_x$;
  \item $[x_0,...,x_k]((1-t)s+t(0,...,0,1))=[[x_0,...,x_{k-1}](s),x_k](t)$ for $s\in\Delta^{k-1}$  
 \end{itemize}
\end{defi}

To prove that the definition is correct it is enough to prove the following lemma.

\begin{lemma}\label{diam}
Let $k\in\mathbb{N}$ and $R<\injr{8k}$. If $x_0,...,x_k\in V_x(R)$ then $[x_0,...,x_k]$ exists and
$$
[x_0,...,x_k](\Delta^k) \subset V_x((k+1)R).
$$
\end{lemma}

\begin{proof}
We prove the statement by induction. For $k=0$ the existence of a geodesic simplex is obvious and does not require any metric assumptions. For $k>0$ $[x_0,...,x_{k-1}]$ exists by induction hypothesis and $[x_0,...,x_{k-1}]\subset V_x(kR)\subset V_x(\injr{8})$. Consider the geodesic join of maps $[x_0,...,x_{k-1}]:\Delta^{k-1}\rightarrow V_x$ and a constant map sending $\Delta^{k-1}$ to the point $x_k$. Obviously this join has the same image in $V_x$ as $[x_0,...,x_k]$. By Lemma \ref{joinlemma} we get
\begin{eqnarray*}
[x_0,...,x_k](\Delta^k) = [[x_0,...,x_{k-1}],\{x_k\}](\Delta^k) \subset V_x(kR+ R) = V_x((k+1)R) 
\end{eqnarray*}
\end{proof}

The most important fact in this section is a positive curvature analogue of Proposition 2.1 in \cite{LS}.

\begin{prop}\label{Lipschitz_prop}
 Let $Y$ be a compact, smooth manifold (possibly with boundary) and let $f,g:Y\rightarrow V_x$ be two Lipschitz maps such that $(\im(f)\cup \im(g))\subset V_x(C_K)$, where $C_K<\injr{8}$ is a constant depending only on $K$. Then $[f,g]$ has Lipschitz constant depending only on $K$ and the Lipschitz constants for $f$ and $g$. Moreover, $[f,g]$ is smooth ($C^1$) if $f$ and $g$ are smooth ($C^1$).
\end{prop}

To proceed, we need two technical lemmas concerning Riemannian geometry. First is the technical result proved in \cite{LS}, which can be easily applied in our situation.

\begin{lemma}[{\cite[Proposition 2.6]{LS}}]\label{Loeh_positive}
Let $V$ be a complete simply connected Riemannian manifold with $\sec(V)<K$, $K>0$. Then every geodesic simplex $\sigma$ in $V$ such that $\diam(\sigma)<\injr{2}$ is smooth. Further, there is a constant $L>0$ such that every geodesic $k$-simplex $\sigma$ of diameter less than $\injr{4}$ satisfies $\|T_x\sigma\|<L$ for every $x\in\Delta^n$.
\end{lemma}

\begin{lemma}\label{curvature_inequality}
  Consider the geodesic triangle $[x_0,x_1,x_2]$ in $V_x$ such that $x_0,x_1,x_2\in V_x(\injr{48})$. Then there exists a constant $D_K$, depending only on the curvature bound $K$, such that for any $t\in[0,1]$
  $$
  d_{V_x}([x_0,x_2](t),[x_1,x_2](t))\leq D_K d_{V_x}(x_0,x_1)
  $$
\end{lemma}

\begin{proof}
If $x_0=x_1$ it is nothing to prove. If not, consider the extension (in any direction) of $[x_0,x_1]$ to a geodesic of length $\injr{24}$ and denote the endpoints of this geodesic by $x'_0, x'_1$. Such geodesic exists because $B_{V_x}(x_0,\injr{24})\subset V_x(\injr{8})$. Now consider the geodesic triangle $[x'_0,x'_1,x_2]$. Note that
$$
d_{V_x}(x'_0,\bar{x})\leq d_{V_x}(x'_0,x_0)+d_{V_x}(x_0,\bar{x}) < \injr{24} + \injr{48} = \injr{16}.
$$
similarly, $d_V(x'_1,\bar{x})< \injr{16}$, hence by Lemma \ref{diam} we have $[x'_0,x'_1,x_2]\subset V_x(\frac{3\pi}{16\sqrt{K}})$. We can therefore use Lemma \ref{Loeh_positive} to conclude that the diffeomorphic simplex map $\sigma:\Delta^2\rightarrow V_x$ from the standard 2-simplex onto $[x'_0,x'_1,x_2]$ is Lipschitz with constant $L$ independent of $\sigma$. Hence
\begin{eqnarray*}
d_{V_x}([x_0,x_2](t),[x_1,x_2](t)) & \leq & L\cdot d_{\Delta^2}(\sigma^{-1}([x_0,x_2](t)), \sigma^{-1}([x_1,x_2](t)) \\ 
 &\leq & L\cdot d_{\Delta^2}(\sigma^{-1}(x_0), \sigma^{-1}(x_1)) \\
 & = & L\sqrt{2}\frac{d_{V_x}(x_0,x_1)}{\pi/24\sqrt{K}}
\end{eqnarray*}
so one can take $D_K = \frac{24L\sqrt{2K}}{\pi}$.
\end{proof}

\begin{proof}[Proof of Proposition \ref{Lipschitz_prop}]
Put $C_K=\injr{48}$. To prove smoothness in the case $f$ and $g$ are smooth, one can rewrite $[f,g]$ as 
$$
[f,g](y,t) = \exp_{f(y)}(t\cdot \exp^{-1}_{f(y)}(g(y))),
$$
where we use the fact that by Proposition~\ref{inj_rad_prop} if $TV_\rho=\{(y,t)\in TV\;:\; y\in V_x(\rho),\, \|t\|<\rho\}$ then
\begin{eqnarray*}
\exp:TV_{\injr{4}} &\rightarrow& V_x\times V_x \\
\exp_x(t) = \exp(x,t) &\mapsto& (x, \exp_x(t))
\end{eqnarray*}
is a diffeomorphism onto its image.

Now, let $(y,t),(y',t')\in Y\times [0,1]$. We have
\begin{eqnarray*}
d_{V_x}([f,g](y,t),[f,g](y',t'))&\leq& d_{V_x}([f(y),g(y)](t),[f(y),g(y)](t')) \\ &+& d_{V_x}([f(y),g(y)](t'),[f(y'),g(y')](t'))
\end{eqnarray*}
The first term can be easily estimated as follows
$$
d_{V_x}([f(y),g(y)](t),[f(y),g(y)](t'))\leq |t-t'|\cdot d_{V_x}(f(y),g(y))\leq |t-t'|\cdot \diam(\im(f)\cup \im(g)).
$$
Recall that by assumption $(\im(f)\cup \im(g))\subset V_x(\injr{48})$. Therefore the second term can be estimated using Lemma \ref{curvature_inequality}
\begin{eqnarray*}
d_{V_x}([f(y),g(y)](t'),[f(y'),g(y')](t')) &\leq& d_{V_x}([f(y),g(y)](t'),[f(y),g(y')](t')) \\
&+& d_{V_x}([f(y),g(y')](t'),[f(y'),g(y')](t')) \\
&\leq& D_K(d_{V_x}(g(y),g(y'))+d_{V_x}(f(y),f(y')))\\
&\leq& D_K(\Lip(f)+\Lip(g))d_Y(y,y').
\end{eqnarray*}
Finally, we obtain
\begin{eqnarray*}
d_{V_x}([f,g](y,t),[f,g](y',t'))&\leq& 2|t-t'|C_K + D_K(\Lip(f)+\Lip(g))d_Y(y,y')\\
&\leq& (2C_K + D_K(\Lip(f)+\Lip(g)))d_{Y\times[0,1]}((y,t),(y',t'))
\end{eqnarray*}
\end{proof}

\begin{remark}
All the facts above could be stated (possibly with some minor changes in constants used) for any Riemannian manifold $V$ with $\sec(V)<K$ with a distinguished point $\bar{x}\in V$ such that the closure of an open ball $B_V(\bar{x},R)$ is complete for some $R$ and there exists $r<R$ such that every point in $B_V(\bar{x},r)$ has injectivity radius at least $\rho>0$. However, the only examples which are important to us at the moment are spaces $V_x$ for $x\in M$.
\end{remark}

\subsection{The piecewise straightening itself}\label{PSP}
Let $M$ be a complete, $n$-dimensional Riemannian manifold with $\sec(M)<K$, $0<K<\infty$ and let $E_{n,K}=\frac{C_K}{2(n+1)}$, where $C_K$ is a constant from Proposition~\ref{Lipschitz_prop}. Choose a locally finite family $(F_j)_{j\in J}$ of pairwise disjoint Borel subsets of $M$ together with points $z_j\in F_j$ and Borel maps $s_j:F_j\rightarrow V_{z_j}(E_{n,K})$ for $j\in J$, such that
\begin{itemize}
 \item $\bigcup_{j\in J}F_j=M$;
 \item for every $j\in J$ $\diam(F_j) < E_{n,K}$;
 \item for every $j\in J$ $s_j$ is a section of $p_{z_j}$ (i.e. $p_{z_j}\circ s_j = id:F_j\rightarrow F_j$) such that $s_j(z_j)=\bar{z_j}$.
\end{itemize}

A family with properties described above always exists. To see this choose a triangulation of $M$ (which exists because $M$ is Riemannian) and divide every triangle into a locally finite family of disjoint Borel sets with sufficiently small diameters. Also sections $s_j$ for $j\in J$ exist because for $x\in F_j$ a lift of the (not necessarily unique) shortest geodesic joining $z_j$ and $x$ has length $<E_{n,K}$ and one can choose $s_j(x)$ to be an endpoint of one of such lifts in a Borel way.

\begin{defi}
Let $F_j$, $z_j$, $s_j$ for $j\in J$ be as above and let $\pi_U:U\rightarrow M$ be a continuous map such that $B_M(z_j,E_{n,K})\subset \im(\pi_U)$. We call a Borel section $s'_j:F_j\rightarrow U$ of $\pi_U$ \emph{admissible} if there exists a continuous map $v_U:V_{z_j}(E_{n,K})\rightarrow U$ such that $s'_j = v_U\circ s_j$ and $\pi_U\circ v_U = p_{z_j}$, i.e. it fits into the commutative diagram
$$
\xymatrix{
V_{z_j}(E_{n,K})\ar@{.>}[rr]^{v_U} \ar@{>}[rrd]^(0.4){p_{z_j}} && U \ar@{>}[d]^{\pi_U} \\
F_j \ar@{>}[u]^{s_j} \ar@{.>}[urr]^(0.2){s'_j} \ar@{^{(}->}[rr] && M
}
$$
\end{defi}

A motivating example is given by the following lemma

\begin{lemma}\label{admissible_lift_lemma}
 Let $x\in M$ and $x'\in V_x(\injr{4})$. Then there exists a unique $j\in J$ and a unique admissible section
 $$
 s^{x'}_j:F_j\rightarrow B_{V_x}(x',2E_{n,K})
 $$
 with respect to the map $p_x:V_x\rightarrow M$ such that $x'\in s^{x'}_j(F_j)$.
\end{lemma}

\begin{proof}
 Let $y=p_x(x')$, then $y$ is contained in a set $F_j$ for some $j\in J$. By Proposition~\ref{transition_prop} we can compose a canonical section $s_j$ with $I^{-1}_{s_j(y),z_j}:B_{V_{z_j}}(s_j(y),\injr{4})\rightarrow V_y(\injr{4})$ and obtain an admissible section $s'_j:F_j\rightarrow V_y(2E_{n,K})$ with respect to $p_y$ such that $s'_j(y)=\bar{y}$. After the composition of this section with $I_{x',x}:V_y(\injr{4})\rightarrow B_{V_x}(x',\injr{4})$ we obtain an admissible section $s^{x'}_j:F_j\rightarrow B_{V_x}(x',2E_{n,K})$ which satisfies required conditions.
 
 To see the uniqueness of $s^{x'}_j$, let $s'^{x'}_{j'}:F_{j'}\rightarrow B_{V_x}(x',2E_{n,K})$ be another admissible section satisfying the above conditions. Note that $F_j\ni p_x\circ s^{x'}_j(x') = p_x\circ s'^{x'}_{j'}(x') \in F_{j'}$, hence $j=j'$. After composing $s^{x'}_j$ and $s'^{x'}_{j'}$ with $I_{s_j(y),z_j}\circ I^{-1}_{x',x}:B_{V_x}(x',\injr{4})\rightarrow B_{V_{z_j}}(s_j(y), \injr{4})$ and using the admissibility of $s'^{x'}_{j'}$ we obtain sections $s_j,s'_j:F_j\rightarrow V_{z_j}(\injr{4})$ and a map $v:V_{z_j}(E_{n,K})\rightarrow V_{z_j}$ such that $v\circ s_j(y)=s'_j(y)$ and the following diagram commutes 
 $$
 \xymatrix{
  V_{z_j}(E_{n,K})\ar@{>}[rr]^{v} \ar@{>}[rrd]^(0.4){p_{z_j}} && V_{z_j} \ar@{>}[d]^{p_{z_j}} \\
  F_j \ar@{>}[u]^{s_j} \ar@{>}[urr]^(0.2){s'_j} \ar@{^{(}->}[rr] && M
  }
 $$
 It suffices to show that $v=Id|_{V_{z_j}(E_{n,K})}$. But since $p_{z_j}$ is a local isometry, $v$ is also, so it is an identity in some neighbourhood of $s_j(y)$. It follows that it must be an identity on the neighbourhood of the geodesic path joining $s_j(y)$ and $\bar{z}_j$, and consequently on every geodesic joining $\bar{z}_j$ with any other point of $V_{z_j}(E_{n,K})$. In consequence $v=Id|_{V_{z_j}(E_{n,K})}$.
\end{proof}

Now we turn back to the definition of piecewise straightening.

\begin{defi}
Let $\sigma:\Delta^k\rightarrow M$ be a Lipschitz singular simplex. Then we say that $\sigma$ is \emph{$\varepsilon$-geodesic} (with respect to $(F_j)_{j\in J}$) if  $\Lip(\sigma) \leq \frac{\varepsilon}{\sqrt{2}}$ and there exists $x\in M$ and a lift $\wt{\sigma}:\Delta^k\rightarrow V_x$ of $\sigma$ such that $\wt{\sigma}$ is geodesic with vertices in some lifts of the points $z_j$, $j\in J$.
\end{defi}

Note that by Proposition~\ref{transition_prop} if $\varepsilon < \injr{4}$ then the above definition does not depend on the choice of $x\in M$ unless $\sigma(\Delta^k)\subset B_M(x,\injr{4})$.

\begin{defi}
Let $\sigma:\Delta^k\rightarrow M$ be a singular simplex and let $S^{(m)}(\sigma)=\sum_i \sigma_i$ be its $m$-times iterated barycentric subdivision, where $m\in\mathbb{N}$. We say that $\sigma$ is \emph{($m$-)piecewise straight} if every $\sigma_i$ in $S^{(m)}(\sigma)$ is $E_{n,K}$-geodesic (with respect to $(F_j)_{j\in J}$).

We say that a (locally finite) chain $c=\sum_{i\in\mathcal{I}}a_i\sigma_i\in C_*(M)$ is \emph{piecewise straight} if there exists $m\in\mathbb{N}$ such that every $\sigma_i$, $i\in\mathcal{I}$, is $m$-piecewise straight.
\end{defi}

Let $\sigma:\Delta^k\rightarrow M$ for $k\leq n$ be a Lipschitz singular simplex. We define the \emph{straightening} of $\sigma$ (with respect to $(F_j)_{j\in J}$) as follows. Choose $m\in\mathbb{N}$ such that each simplex $\sigma_i$ in $S^{(m)}(\sigma) = \sum_i\sigma_i$ has Lipschitz constant less than $\frac{E_{n,K}}{\sqrt{2}}$. Such $m$ exists because diameters of subdivided simplices in $\Delta^k$ tends to $0$ \cite[Corollary 9.4.9]{ES}, hence also Lipschitz constants of subdivided simplices in $\sigma$. Moreover, we can choose $m$ depending only on $n$, $K$ and $\Lip(\sigma)$. For every simplex $\sigma_i$ choose a point $y_i\in\Delta^k$ and let $y'_i=\sigma_i(y_i)$, then by Corollary~\ref{simplex_lift_cor} there is a unique lift $\wt{\sigma_i}:\Delta^k\rightarrow V_{y'_i}(E_{n,K})$ of $\sigma_i$ such that $\wt{\sigma_i}(y_i)=\bar{y}'_i$. Denote by $x_{i,0},...,x_{i,k}$ its vertices, let $s'_{i,l}:F_{i,l}\rightarrow V_{y'_i}$ be admissible sections containing $x_{i,l}$ in their images for $l=0,...,k$, constructed by Lemma~\ref{admissible_lift_lemma} and let $z'_{i,l} = s'_{i,l}(z_{i,l})$ for $l=0,...,k$. In particular $z'_{i,0},...,z'_{i,k}\in V_{y'_i}(2E_{n,K})$, hence the geodesic simplex $[z'_{i,0},...,z'_{i,k}]$ exists by Lemma~\ref{diam} because
$$
2E_{n,K} =\frac{2C_K}{2(n+1)} < \injr{8(n+1)} < \injr{8k}.
$$
Let $\str_{y_i}(\sigma_i) = [z'_{i,0},...,z'_{i,k}]$ and define
$$
\str_m(\sigma) = (S^{(m)})^{-1}(\sum_i p_{y'_i}\circ\str_{y_i}(\sigma_i)).
$$
Moreover, by Lemma~\ref{diam} we have
$$
[z'_{i,0},...,z'_{i,k}] \subset V_{y'_i}(2(k+1)E_{n,K})\subset V_{y'_i}(C_K)
$$
so it follows from Proposition~\ref{Lipschitz_prop} that $[\wt{\sigma_i},[z'_{i,0},...,z'_{i,k}]]$ exists and defines a Lipschitz homotopy $H_{y_i}:\Delta^k\times I\rightarrow V_{y'_i}$ between these simplices, with Lipschitz constant depending only on $m$, $K$ and $\Lip(\sigma)$. Define 
$$
H_m (\sigma) = (S^{(m)}\times Id_I)^{-1}(\sum_i p_{y'_i}\circ H_{y_i}(\sigma_i)).
$$

To show that $\str_m$ and $H_m$ are well defined it suffices to verify that the construction is independent of the choice of $y_i\in\Delta^k$. Indeed, assuming this fact if $\partial^q:C_k(M)\rightarrow C_{k-1}(M)$ for $q=0,...,k$ is an operator assigning to a singular simplex its $q$-th face, we see that for any $\dot{y}_i\in \Delta^{k-1}\subset\partial^q\Delta^k$ and $\dot{y}'_i=\sigma_i(\dot{y}_i)$ we have
$$
\partial^q(p_{y'_i}\circ\str_{y_i}(\sigma_i)) = \partial^q(p_{\dot{y}'_i}\circ\str_{\dot{y}_i}(\sigma_i)) = p_{\dot{y}'_i}\circ\partial^q\str_{\dot{y}_i}(\sigma_i) = p_{\dot{y}'_i}\circ\str_{\dot{y}_i}(\partial^q\sigma_i).
$$
where the last equality is the consequence of the fact that the straightening of a face of any singular simplex depends only on this particular face, not on the whole simplex. In particular, if two simplices $\sigma_i$ and $\sigma_{i'}$ have some face in common, their straightenings will also have the same one. This shows that $\sum_i p_{y'_i}\circ\str_{y_i}(\sigma_i)$ lies in the image of $S^{(m)}$, hence (after giving some ordering on the vertices of $\sigma_i$) we can choose a preimage in the canonical way. The same proof applies also to $H_m$.

Now we verify our claim. Let $\dot{y}_i\in\Delta^k$, $\dot{y}'_i=\sigma_i(\dot{y}_i)\in M$ and $\wt{y}'_i=\wt{\sigma}_i(\dot{y}'_i)\in V_{y'_i}(E_{n,K})$. Then Proposition~\ref{transition_prop} gives an isometry $I_{\wt{y}'_i,y'_i}$ between $V_{\dot{y}'_i}(\injr{4})$ and $B_{V_{y'_i}}(\wt{y}'_i,\injr{4})$.
By Lemma~\ref{joinlemma} we have
$$
H_{y_i}(\Delta^k\times I)\subset V_{y'_i}(C_K+E_{n,K}) \subset V_{y'_i}(\frac{3\pi}{16\sqrt{K}}).
$$
and because $d_{V_{y'_i}}(\bar{y}'_i,\wt{y}'_i)<E_{n,K}<\injr{16}$, the images of $H_{y_i}$ and $H_{\dot{y}_i}$ stay in $B_{V_{y'_i}}(\wt{y}'_i,\injr{4})$ and $V_{\dot{y}'_i}(\injr{4})$ respectively. Moreover, $I_{\wt{y}'_i,y'_i}$ maps respective admissible sections $\dot{s}'_{i,l}:F_{i,l}\rightarrow V_{\dot{y}'_i}$ to admissible sections $s'_{i,l}:F_{i,l}\rightarrow V_{y'_i}$, hence $H_{y_i}=I_{\wt{y}'_i,y'_i}\circ H_{\dot{y}_i}$. As a result they are the same after pushing them back on $M$. This argument applies also to $\str_{y_i}$ since $\str_{y_i} = H_{y_i}(-,1)$.

Let $c=\sum_ia_i\sigma_i$ be a locally finite Lipschitz chain with Lipschitz constant $L$. We see that we can choose $m\in\mathbb{N}$, depending only on $n$, $L$ and $K$, such that $\str_m(\sigma_i)$ is defined for every $i$, so we can define $\str_m(c)$ simply as $\sum_i a_i \str_m(\sigma_i)$. The chain $\str_m(c)$ is Lipschitz because of Proposition~\ref{Lipschitz_prop} and Lemma~\ref{lemma_geodesic_lip}, and locally finite since by construction for any singular simplex $\sigma:\Delta^k\rightarrow M$ we have $\str_m(\sigma)\subset B_M(\sigma(\Delta^k),C_K)$. Note that the straightening defined as above does not define a chain operator $C^{lf,\Lip}_{*\leq n}(M)\rightarrow C^{lf,\Lip}_{*\leq n}(M)$, where $C^{lf,\Lip}_*(M)$ are locally finite Lipschitz chains, because we cannot choose $m$ uniformly. However, it allows us to prove slightly weaker statement. Recall that $C^{lf,<L}_*(M)$ is a chain complex of locally finite singular chains on $M$ consisting of simplices with Lipschitz constant less than $L$.

\begin{lemma}\label{chain_homotopy_lemma}
 For every $L<\infty$ there exists $m\in\mathbb{N}$ such that the operator
 $$
 \str_m:C^{lf,<L}_{*\leq n}(M)\rightarrow C^{lf,\Lip}_*(M) 
 $$
 is a well defined chain map homotopic to the inclusion $\iota:C^{lf,<L}_{*\leq n}\rightarrow C^{lf,\Lip}_*(M)$. Moreover, $|\str_m|_1\leq 1$.
\end{lemma}

\begin{proof}
 Choose $m$ such that $\str_m$ is well defined for any singular simplex $\sigma:\Delta^k\rightarrow M$ with $k\leq n$ and $\Lip(\sigma)<L$. Then for any such singular simplex $\sigma$ let $S^{(m)}(\sigma)=\sum_i\sigma_i$. For arbitrary $y\in\Delta^k$ and $y^q\in\partial^q\Delta^k$, $q=0,...,k$, we have
 \begin{eqnarray*}
 \partial\str_m(\sigma) & = & \sum_{q=0}^k(-1)^q\partial^q (S^{(m)})^{-1}(\sum_i p_{\sigma_i(y)}\circ\str_y(\sigma_i)) \\
 & = & \sum_{q=0}^k(-1)^q (S^{(m)})^{-1}(\sum_i \partial^q (p_{\sigma_i(y)}\circ\str_y(\sigma_i)))\\
 & = & \sum_{q=0}^k(-1)^q (S^{(m)})^{-1}(\sum_i p_{\partial^q\sigma_i(y^q)}\circ\str_{y^q}(\partial^q\sigma_i))\\
 & = & \str_m(\sum_{q=0}^k(-1)^q \partial^q\sigma) = \str_m(\partial\sigma)
 \end{eqnarray*}
 where we use the fact that $S^{(m)}$ is a chain operator and the construction of $\str_m$. This shows that $\str_m$ is a chain map. To obtain a chain homotopy joining $\str_m$ and $\iota$ let $P_k\in C_{k+1}(\Delta^k\times I)$ be a canonical division of $\Delta\times I$ into singular simplices described e.g. in \cite[Proof of 2.10]{HAT} and let $h:C_k^{lf,<L}(M)\rightarrow C_{k+1}^{lf,\Lip}(M)$ for $k\leq n$ be defined as $h(\sigma) = H_m(\sigma)_*(P_k)$. Note that $h(c)$ is Lipschitz, because $H_m$ is by Proposition~\ref{Lipschitz_prop} and Lemma~\ref{lemma_geodesic_lip}, and locally finite since by construction $H_m(\sigma)(\Delta^k\times I)\subset B_M(\sigma)(\injr{4})$ for any $\sigma:\Delta^k\rightarrow M$. The proof that it provides the desired chain homotopy is standard and described e.g. in \cite[Proof of Theorem 2.10]{HAT} or \cite[Lemma 2.13]{LS}. The proof that $|\str_m|_1\leq 1$ is straightforward.
\end{proof}

\begin{cor}\label{homology_representation}
Every homology class $\xi$ in $H_{*\leq n}^{lf,\Lip}(M)$ can be represented by a piecewise straight chain with vertices in $(z_j)_{j\in J}$. Moreover, $l^1$ semi-norm on $H_{*\leq n}^{lf,\Lip}(M)$ can be computed on piecewise straight chains. 
\end{cor}

\begin{proof}
Let $c=\sum_i a_i\sigma_i\in C_k^{lf,\Lip}(M)$, $k\leq n$, be any cycle such that $[c]=\xi$. Then $c\in C_k^{lf,<L}(M)$ for some $L<\infty$. Hence by Lemma~\ref{chain_homotopy_lemma} there exists $m$ such that a chain $\str_m(c)$ is homologuous to $c$ and $|\str_m(c)_1|\leq |c|_1$. It is also obviously straight and has its vertices in $(z_j)_{j\in J}$.
\end{proof}

\begin{remark}
The results above are stated only for $*\leq n$. However, for $*>n$ groups $H_*^{lf,\Lip}(M)$ vanish \cite[Theorem 3.3]{LS}. Moreover, we could simply modify the constants used in the straightening to work for $*\leq N$ for $N$ arbitrarily large. In further work we will without loss of generality assume that all chains and homology classes are of dimension $*\leq n$.
\end{remark}

\begin{remark}
It is obvious that the straightening procedure depends on the choice of sets $(F_j)_{j\in J}$, sections $s_j$ for $j\in J$ and $m\in\mathbb{N}$, which depends on particular chain which we would like to straighten. However, in most cases these details are of secondary interest, therefore we will just say shortly about \emph{applying (piecewise) straightening procedure} meaning applying it with respect to any suitable family $(F_j)_{j\in J}$ and any $m\in\mathbb{N}$ for which the procedure is defined.
\end{remark}

\section{Piecewise $C^1$ homology theories}\label{SecHom}
The straightening procedure described in the previous section is sufficient for some applications, though we need some more extensive machinery. One of the key properties of the classical straightening procedure for non-positively curved manifolds is that the straightened chains are smooth, because they consist of geodesic simplices. It is important e.g. in the proof of the proportionality principle in non-positively curved case, which depends on measure homology with $C^1$ Lipschitz support, i.e. where 'chains' are Borel measures with finite variation on $C^1$ singular simplices with $C^1$-topology, with additional assumption that support of each 'chain' is contained in $L$-Lipschitz simplices for some $L<\infty$. Differentability here is strictly technical, but necessary, because it allows to recognise fundamental cycle by integrating the volume form. However, the piecewise straight simplices which we use are only piecewise $C^1$.

In Section~\ref{PSH} we define piecewise $C^1$ simplices and chains and introduce piecewise $C^1$ homology. In Section~\ref{PSMH} we provide some reasonable topology on these simplices in order to define corresponding measure homology theory.

\subsection{Piecewise $C^1$ homology}\label{PSH}
Let $M$ be a connected, complete $n$-dimensional Riemannian manifold with $\sec(M)<K$. Before we continue, let us fix some notation concerning convex polyhedra. Let $V\subset\mathbb{R}^n$ be an affine space and let $\langle,\rangle$ be a truncation of a standard scalar product on $\mathbb{R}^n$ to $V$.
\begin{itemize}
 \item for $v\in V$ and $b\in\mathbb{R}$ a \emph{half-space} $H_{v,b}\subset V$ is
 $$
 H_{v,b}=\{x\in V\::\:\langle x,v\rangle\leq b\};
 $$
 \item a convex polyhedron $P\subset V$ is an intersection of finite number of half-spaces;
 \item $\dim{P} = \min\{\dim{W}\::\: P\subset W\,,\,W\subset V \text{ is an affine subspace}\}$;
 \item $P$ is \emph{nondegenerated} if $\dim{P} = \dim{V}$;
 \item for a convex polyhedron $P$ a map $f:P\rightarrow M$ is $C^1$ if it can be extended to a $C^1$ map $f':U\rightarrow M$, where $U\subset V$ is some open neighbourhood of $P\subset V$.
\end{itemize}

\begin{defi}
Let $V =\{(x_0,...,x_k)\in \mathbb{R}^{k+1}\::\: \sum_{i=0}^k x_i=1\}\supset \Delta^k$. We say that a family $\mathcal{P}$ of nondegenerated convex polyhedra $P\subset\Delta^k$ is \emph{$\Delta^k$-admissible} if it satisfies
\begin{itemize}
 \item $\bigcup_{P\in\mathcal{P}} P = \Delta^k$;
 \item $\forall_{P_1,P_2\in\mathcal{P}}\: P_1\neq P_2\Rightarrow \dim{P_1\cap P_2}<k$.
\end{itemize}

We will denote the family of all $\Delta^k$-admissible families by $\mathscr{P}_k$.
\end{defi}

A good example of a $\Delta^k$-admissible family of convex polyhedra is a barycentric subdivision $S\Delta^k$ and more generally $m$-times iterated barycentric subdivision $S^{(m)}\Delta^k$.

For two families $\mathcal{P}_1, \mathcal{P}_2\in \mathscr{P}_k$ we can define their product
$$
\mathcal{P}_1\cdot\mathcal{P}_2 = \{P_1\cap P_2\::\:P_1\in\mathcal{P}_1,\,\:P_2\in\mathcal{P}_2\,,\,\dim{P_1\cap P_2}=k\}
$$
which is also $\Delta^k$-admissible. This product is obviously commutative and associative. Moreover, we can put a partial order on $\mathscr{P}_k$ by
$$
\mathcal{P}_1\leq \mathcal{P}_2\Leftrightarrow \forall_{P_2\in\mathcal{P}_2}\exists_{P_1\in\mathcal{P}_1}\:P_2\subset P_1.
$$
Note that with this order every finite set $\{\mathcal{P}_1,...,\mathcal{P}_m\}\subset \mathscr{P}_k$ has supremum $\mathcal{P}_1\cdot...\cdot\mathcal{P}_m$.

\begin{defi}
Let $\mathcal{P}$ be a $\Delta^k$-admissible family of convex polyhedra and let $\sigma:\Delta^k\rightarrow M$ be a singular simplex. We say that it is \emph{$\mathcal{P}$-$C^1$} if for every $P\in\mathcal{P}$ $\sigma|_P:P\rightarrow M$ is of class $C^1$.

A chain $c\in C^{lf}_k(M)$ is called \emph{$\mathcal{P}$-$C^1$} if it consists of $\mathcal{P}$-$C^1$ simplices and is \emph{piecewise $C^1$} if it is $\mathcal{P}$-$C^1$ for some $\mathcal{P}\in\mathscr{P}_k$.
\end{defi}

Note that if $c_1, c_2\in C^{lf}_k(M)$ are singular chains such that $c_1$ is $\mathcal{P}_1$-$C^1$ and $c_2$ is $\mathcal{P}_2$-$C^1$ then $c_1+c_2$ is $\mathcal{P}_1\cdot\mathcal{P}_2$-$C^1$, hence finite sums of piecewise $C^1$ chains are piecewise $C^1$. Moreover, if $c\in C_k^{lf}(M)$ is $\mathcal{P}$-$C^1$ then $\partial c$ is $\prod_{q=0}^k\partial^q\mathcal{P}$-$C^1$, where
$$
\partial^q\mathcal{P} = \{P\cap \partial^q\Delta^k\::\:P\in\mathcal{P}\,,\,\dim{P\cap \partial^q\Delta^k}=k-1\}
$$
for $q=0,...,k$. In particular piecewise $C^1$ chains form a subcomplex of $C^{lf}_*(M)$. The same is true for $C^{lf,\Lip}_*(M)$ if we consider piecewise $C^1$ Lipschitz chains. Therefore we can define \emph{piecewise $C^1$ locally finite homology} $H^{PC^1,lf}_*(M)$ and \emph{piecewise $C^1$ locally finite Lipschitz homology} $H_*^{PC^1,lf,Lip}(M)$.

Obviously every piecewise straight chain is piecewise smooth (with respect to some iterated barycentric subdivision) by Lemma \ref{Loeh_positive}. To show that the homology theories defined above are isometric to the corresponding non-$C^1$ ones, we need the following lemma.

\begin{lemma}\label{piecewise_homotopy_lemma}
 Let $c\in C_k^{PC^1,lf,\Lip}(M)$ be a piecewise $C^1$ locally finite Lipschitz cycle and let $m\in\mathbb{N}$ be such that $\str_m(c)$ is defined. Then $c$ and $\str_m(c)$ are homologous in $C_{k}^{PC^1,lf,\Lip}(M)$.
\end{lemma}

\begin{proof}
 In the proof of Lemma \ref{chain_homotopy_lemma} we constructed a chain homotopy $h:C_k^{lf,<L}(M)\rightarrow C_{k+1}^{lf,\Lip}(M)$ between an identity and $\str_m$ for some $L<\infty$. Therefore it suffices to to show that if $c$ is piecewise $C^1$, then $h(c)$ is.

 Assume that a singular simplex $\sigma$ is $\mathcal{P}$-$C^1$. Then a map $H_m(\sigma)|_{P\times I}:P\times I\rightarrow M$ is $C^1$ for $P\in \mathcal{P}\cdot(S^{m}\Delta^k)$. Moreover, if $P_k$ is a canonical subdivision of a prism as in \cite[Proof of 2.10]{HAT} (this time considered as a set of simplices) then for any simplex $\Delta'\in P_k$ a family
 $$
 \mathcal{P}_{m,\Delta'}=\{\Delta'\cap (P\times I)\::\: P\in\mathcal{P}\cdot(S^{m}\Delta^k)\,,\,\dim{(\Delta'\cap (P\times I))} = k+1 \}
 $$
 is (up to some affine isomorphism) $\Delta^{k+1}$-admissible and $H_m(\sigma)$ is $C^1$ on every $P\in\mathcal{P}_{m,\Delta'}$, hence $h(\sigma)$ is $\prod_{\Delta'\in P_k}\mathcal{P}_{m,\Delta'}$-$C^1$. In particular $h(c)$ is piecewise $C^1$ if $c$ is.
 \end{proof}

\begin{prop}\label{homology_theories}
 Let $M$ be a complete Riemannian manifold with $\sec(M)<K<\infty$. Then the map $I_*:H^{PC^1,lf,\Lip}_*(M)\rightarrow H^{lf,\Lip}_*(M)$ induced by the inclusion of chains is an isometric isomorphism.
\end{prop}

\begin{proof}
 The map $I_*$ is onto by Corollary~\ref{homology_representation}. To see that it is injective consider $c_1,c_2\in C_*^{PC^1,lf,\Lip}(M)$ which represent the same class in $H^{lf,\Lip}_*(M)$. Then there exists a chain $D\in C_{*+1}^{lf,\Lip}(M)$ such that $\partial D=c_2-c_1$. We can apply now the straightening procedure to $D$ to obtain a chain $\str_m(D)\in C_{*+1}^{PC^1,lf,\Lip}(M)$ for some $m\in\mathbb{N}$ such that $\partial \str_m(D) = \str_m(c_2)-\str_m(c_1)$. Now apply Lemma~\ref{piecewise_homotopy_lemma} to see that $c_1$ and $c_2$ are homologous (in $C_*^{PC^1,lf,\Lip}(M)$) to $\str_m(c_1)$, $\str_m(c_2)$ respectively. It is an isometry on homology by Corollary~\ref{homology_representation}.
\end{proof}

\subsection{Piecewise $C^1$ Measure Homology}\label{PSMH}
Now we turn our attention to chains with finite $\ell^1$-norm and corresponding measure homology theory.

\begin{defi}
 Let $C^{PC^1, \ell^1, \Lip}_*(M)$ be a chain subcomplex of $C^{PC^1,lf,\Lip}_*(M)$ consisting of chains which have finite $\ell^1$ norm. We call the homology of this complex \emph{piecewise $C^1$-$\ell^1$ Lipschitz homology} and denote it by $H_*^{PC^1,\ell^1,\Lip}(M)$.
\end{defi}

\begin{remark}
 Note that Lemma~\ref{piecewise_homotopy_lemma} also applies to $C^{PC^1, \ell^1, \Lip}_*(M)$, so an analogue of Proposition~\ref{homology_theories} for $H_*^{PC^1,\ell^1,\Lip}(M)$ is true.
\end{remark}

\begin{defi}
 Let $\mathcal{P}\in\mathscr{P}_k$ be a $\Delta^k$-admissible family and let $\mathcal{P}C^1(\Delta^k, M)$ be a set of singular simplices $\sigma:\Delta^k\rightarrow M$ such that $\sigma|_P$ is $C^1$ for every $P\in\mathcal{P}$. We call it the set of \emph{$\mathcal{P}$-$C^1$ singular simplices}. We equip it with the topology induced from the embedding onto a closed subspace
 $$
 \mathcal{P}C^1(\Delta^k, M)\rightarrow \prod_{P\in\mathcal{P}}C^1(P,M).
 $$
 Where $C^1(P,M)$ is a set of $C^1$ maps $P\rightarrow M$ with the topology induced from $Map(TP,TM)$ with compact-open topology. For every $\mathcal{P}_1,\mathcal{P}_2\in\mathscr{P}_k$ such that $\mathcal{P}_1\leq \mathcal{P}_2$ we have an embedding $\mathcal{P}_1C^1(\Delta^k, M)\rightarrow \mathcal{P}_2C^1(\Delta^k, M)$ onto a closed subset. We denote the direct limit of these spaces with weak topology as $\mathscr{P}C^1(\Delta^k,M)$.
\end{defi}

The properties of the above topology on $\mathcal{P}C^1(\Delta^k, M)$ for $\mathcal{P}\in\mathscr{P}_k$ which are crucial to us are the following
\begin{itemize}
 \item $\mathcal{P}C^1(\Delta^k, M)$ is a locally compact Hausdorff space;
 \item If $k=n=\dim M$ then for every differential form $\omega\in\Omega^n(M)$ the map
  $$
    I_\omega:\mathcal{P}C^1(\Delta^n,M)\rightarrow \mathbb{R}\;,\; f \mapsto \int_{\Delta^n} f^*\omega
  $$
 is continuous;
 \item for every $\sigma\in \mathcal{P}C^1(\Delta^k, M)$ the map
 $$
 Isom^+(M)\rightarrow \mathcal{P}C^1(\Delta^k, M)\;,\; g \mapsto g\sigma
 $$
 is continuous.
\end{itemize}

\begin{defi}
Let $\mathcal{C}_*^{PC^1,\Lip}(M)$ be a chain complex of measures on $\mathscr{P}C^1(\Delta^*,M)$ such that
\begin{enumerate}
 \item for every measure $\mu\in\mathcal{C}_*^{PC^1,\Lip}(M)$ there exists $\mathcal{P}\in\mathscr{P}_k$ such that it is a push-forward of a Borel measure on $\mathcal{P}C^1(\Delta^*,M)$ with finite variation;
 \item every measure has \emph{Lipschitz determination}, i.e. there exists $L<\infty$ such that it is supported on simplices with Lipschitz constant $L$.
\end{enumerate}
The boundary operators are the push-forwards of measures by the boundary maps $\partial:C_*^{PC^1,\ell^1,\Lip}(M)\rightarrow C^{PC^1,\ell^1,\Lip}_{*-1}(M)$. The obtained homology theory is called \emph{piecewise $C^1$ measure homology with Lipschitz determination} $\mathcal{H}_*^{PC^1,\Lip}(M)$.
\end{defi}

\begin{remark}
 The space $\mathscr{P}C^1(\Delta^*,M)$ is in general not locally compact, therefore it is a problem with the definition of Borel measures. However, we will say for simplicity that measures in $\mathcal{C}_*^{PC^1,\Lip}(M)$ are Borel meaning that every such measure is a push-forward of a Borel measure on $\mathcal{P}C^1(\Delta^*,M)$ for some $\mathcal{P}\in\mathscr{P}_k$. Similarly, when integrating over $\mathscr{P}C^1(\Delta^*,M)$, we will understand it as an integral over $\mathcal{P}C^1(\Delta^*,M)$ for some 'sufficiently large' $\mathcal{P}\in\mathscr{P}_k$.
\end{remark}
 
The above homology theory is a variant of Milnor-Thurston homology. We can introduce a semi-norm $\|\cdot\|_1$ on it by taking infimum of absolute variations over all measures representing given homology class. An important consequence of the above construction is the following

\begin{prop}\label{measure_hom_theories}
 Let $M$ be a complete Riemannian manifold with $\sec(M)<K<\infty$. Then the homology groups $H_*^{PC^1, \ell^1,\Lip}(M)$ and $\mathcal{H}_*^{PC^1,\Lip}(M)$ are isometrically isomorphic.
\end{prop}

\begin{proof}
 By interpreting singular chains with finite $\ell^1$-norm as discrete measures with finite variation we have an obvious inclusion of chains $I:C_*^{PC^1,\ell^1,\Lip}(M)\rightarrow \mathcal{C}_*^{PC^1,\Lip}(M)$ which commutes with differentials, hence it is a morphism of chain complexes and induces a homomorphism $I_*$ on homology.
 
 To show that $I_*$ is surjective, let $\mu\in\mathcal{C}_*^{PC^1,\Lip}(M)$ be a measure cycle determined in $L$-Lipschitz simplices. Choose any family $(F_j)_{j\in J}$ of Borel subsets of $M$ with the properties indicated in the description of the piecewise straightening procedure and $m\in\mathbb{N}$ such that $\str_m$ is defined for any simplex with Lipschitz constant $L$. Then after applying $\str_m$ to the measure $\mu$ we obtain a cycle $\sum_{i\in I}a_i\sigma_i$, where each $\sigma_i, i\in I$ is an $m$-piecewise straight simplex and
 $$
 a_i=\mu(\{\sigma\in \mathscr{P}C^1(\Delta^*,M);\; \Lip(\sigma)\leq L,\; \str_m(\sigma)=\sigma_i\}).
 $$
 The subset of $\mathscr{P}C^1(\Delta^*,M)$ described above is Borel by the construction of sets $(F_j)_{j\in J}$, so the cycle is well defined. It is also piecewise smooth by Proposition~\ref{Lipschitz_prop}, locally finite by the local finiteness of $(F_j)_{j\in J}$ and Lipschitz by Proposition~\ref{Lipschitz_prop} and Lipschitz determination of $\mu$. It is also easy to see that $\mu$ and $\str_m(\mu)$ are homologous in $\mathcal{C}_*^{b1,\Lip}(M)$ by the same argument as in Lemma~\ref{piecewise_homotopy_lemma}.
 
 The injectivity of $I_*$ can be shown using the similar argument applied to the homotopy in $\mathcal{C}^{PC^1,\Lip}_*(M)$ between two cycles in $C_*^{PC^1,\ell^1,\Lip}(M)$ and Lemma~\ref{piecewise_homotopy_lemma}. The fact that $I_*$ is an isometry is a consequence of the facts that $I$ is an isometric inclusion and that the straightening procedure does not increase the norm.
\end{proof}

\begin{remark}
  The existence of an isometric isomorphism as above for 'finite' piecewise $C^1$ theory $H^{PC^1}_*(M)$ and piecewise $C^1$ measure homology with compact supports $\mathcal{H}^{PC^1}_*(M)$ can be proved without any curvature assumptions as in \cite{LoM}. However, the proof given in \cite{LoM} depends heavily on bounded cohomology and cannot be easily generalised to the locally finite Lipschitz case.
\end{remark}

\section{Applications}\label{SecApp}

\subsection{Product inequality}
There is a classical result concerning the behaviour of simplicial volume under taking products. Namely, if $M$ and $N$ are compact manifolds of dimensions $m$ and $n$ respectively there are inequalities (see \cite{G} for more details)
$$
\|M\|\cdot\|N\| \leq \|M\times N\| \leq {m+n \choose m}\|M\|\cdot\|N\|.
$$
The second inequality is obtained by simply taking a simplicial approximation of a cross product and can be easily generalised to the noncompact case. On the other hand, first inequality can be established by passing to bounded cohomology and using the duality between $\ell^1$ semi-norm on homology and $\ell^\infty$ semi-norm on cohomology. However, this approach does not generalize directly to the case of noncompact manifolds and Lipschitz simplicial volume (and in general is false in noncompact, non-Lipschitz case). Two main problems which arise are more subtle relation between $\ell^1$ semi-norm on locally finite homology and $\ell^\infty$ semi-norm on cohomology with compact supports and the existence of a good product in cohomology with compact supports. However, for the Lipschitz simplicial volume the inequality was proved in the case of complete, non-positively curved Riemannian manifolds \cite[Theorem 1.7]{LS}. Using piecewise straightening procedure, we are able to generalize it slightly and obtain Theorem~\ref{PI1}.

The proof is a modification of the proof from \cite{LS} with one proposition generalised to the case of bounded positive curvature. We introduce necessary notions and facts. By $S_k^{lf,\Lip}(M)$ we denote a family of subsets of $Map(\Delta^k,M)$ such that each element $A\in S_k^{lf,\Lip}$ is locally finite (in the sense that for any given compact subset $K\subset M$ we have $\#\{\sigma\in A\;:\; \sigma\cap K\neq \emptyset \}<\infty$) and consists of $L$-Lipschitz simplices for some $L$, depending on $A$. We recall the most important definitions and results from \cite{LS}.

\begin{defi}[{\cite[Definition 3.11]{LS}}]
 Let $M$ be a topological space, $k\in\mathbb{N}$, and let $A\subset Map(\Delta^k,M)$.
 \begin{enumerate}
  \item For a locally finite chain $c=\sum_{i\in I}a_i\sigma_i\in C^{lf}_k(M)$, let
  $$
  |c|^A_1 = \begin{cases}|c|_1 & \text{if $\supp(c)\subset A$,} \\ \infty & \text{otherwise.}\end{cases}
  $$
  Here, $\supp(c):=\{\sigma_i;\; i\in I,a_i\neq 0\}$.
  
  \item The semi-norms on (Lipschitz) locally finite homology induced by $|\cdot|^A_1$ are denoted by $\|\cdot\|^A_1$.
  
  \item If $M$ is an oriented, connected $n$-manifold, then
  $$
  \|M\|^A:=\|[M]\|^A_1.
  $$
  \end{enumerate}
\end{defi}

\begin{defi}[{\cite[Definition 3.19]{LS}}]
 Let $M$ and $N$ be two topological spaces, and let $k,l\in\mathbb{N}$. A~locally finite set $A\in S^{lf}_{k+l}(M\times N)$ is called \emph{$(k,l)$-sparse} if
 $$
 \begin{matrix}
 A_M: = \{\pi_M\circ\sigma\rfloor_k;\; \sigma\in A\}\in S^{lf}_k(M) & \text{     and     } & A_N:=\{\pi_N\circ{}_l\lfloor\sigma;\; \sigma\in A\}\in S^{lf}_l(N)
 \end{matrix}
 $$
 where $\sigma\rfloor_k$ is an $k$-dimensional face of $\sigma$ spanned by the last $k$ vertices, ${}_l\lfloor\sigma$ is an $l$-dimensional face of $\sigma$ spanned by the first $l$ vertices and $\pi_M:M\times N\rightarrow M$ and $\pi_N:M\times N\rightarrow N$ are the canonical projections.
 
 A locally finite chain $c\in C^{lf}_{k+l}(M\times N)$ is called \emph{$(k,l)$-sparse} if its support is $(k,l)$-sparse.
\end{defi}

The proof of product inequality given in \cite{LS} uses non-positive curvature only to prove that for two non-positively curved manifolds the simplicial volume can be computed on sparse cycles. The following proposition is not stated as such in \cite{LS}, it is, however, a meta-theorem actually proved there.

\begin{prop}
 Let $M$ and $N$ be two complete, oriented manifolds of dimensions $m$ and $n$ respectively such that the Lipschitz simplicial volume of $M\times N$ can be computed via $(m,n)$-sparse fundamental cycles, i.e.
 $$
 \|M\times N\|_{\Lip} = \inf\{\|M\times N\|^A;\; A\in S^{lf,\Lip}_{m+n}(M\times N),\; \text{$A$ is $(m,n)$-sparse}\}.
 $$
 Then
 $$
 \|M\|_{\Lip}\cdot\|N\|_{\Lip}\leq \|M\times N\|_{\Lip}.
 $$
\end{prop}

The outline of the proof is as follows. Consider cohomology with Lipschitz compact supports $H^*_{cs,\Lip}$, i.e. cohomology of cochain complex consisting of those singular cochains for which there exists a compact set $K$ and constant $L$ such that the evaluation on a simplex $\sigma$ is $0$ if $\sigma$ has image disjoint from $K$ and Lipschitz constant less than $L$. For a given space $X$ and family $A\in S^{lf,\Lip}_k(X)$ an $\ell^\infty$ semi-norm $\|\cdot\|^A_\infty$ on cohomology with Lipschitz compact support computed on $A$ is dual to the $\ell_1$ semi-norm $\|\cdot\|_1^A$ on Lipschitz locally finite homology \cite[Proposition 3.12]{LS}. Therefore one can compute a Lipschitz simplicial volume using a dual point of view. Moreover for two cochains with Lipschitz compact supports on $M$ and $N$ we can define their product on $M\times N$  which also has Lipschitz compact support \cite[Lemma 3.15]{LS}. Finally, if $A\in S^{lf,\Lip}_{m+n}(M\times N)$ is $(m,n)$-sparse and $A_M$, $A_N$ are corresponding projections of $A$ to $S^{lf,\Lip}_m(M)$ and $S^{lf,\Lip}_n(N)$, then we have a product inequality \cite[Remark 3.17]{LS}:
$$
\|\phi\times\psi\|^A_\infty \leq \|\phi\|^{A_M}_\infty \cdot \|\psi\|^{A_N}_\infty
$$
for $\phi\in H^m_{cs,\Lip}(M)$ and $\psi\in H^n_{cs,\Lip}(N)$. In particular if $A$ is $(m,n)$-sparse we obtain
$$
\|M\times N\|^A=\frac{1}{\|[M\times N]^*\|^A_\infty}\geq \frac{1}{\|[M]^*\|^{A_M}_\infty\cdot\|[N]^*\|^{A_N}_\infty}=\|M\|^{A_M}\cdot\|N\|^{A_N}\geq\|M\|_{\Lip}\
\cdot\|N\|_{\Lip},
$$
hence if the simplicial volume of $M\times N$ can be computed on sparse cycles we get the desired inequality.

To finish the proof of Theorem \ref{PI1} we will prove the following proposition, which is a generalization of Proposition 3.20 in \cite{LS}, where it was proved assuming non-positive curvature.

\begin{prop}
 Let $M$ and $N$ be two oriented, connected, complete Riemannian manifolds (without boundary) of dimensions $m$ and $n$ respectively with sectional curvatures bounded from above by $0<K<\infty$.
 \begin{enumerate}
  \item Let $k,l\in\mathbb{N}$. For any cycle $c\in C_{k+l}^{lf,\Lip}(M\times N)$ there is a $(k,l)$-sparse cycle $c'\in C_{k+l}^{lf,\Lip}(M\times N)$ satisfying
  $$
  \begin{matrix}
   |c'|_1\leq |c|_1 & \text{     and     } & c\sim c' \text{ in } C_{k+l}^{lf,\Lip}(M\times N).
  \end{matrix}
  $$
  \item In particular, the Lipschitz simplicial volume can be computed via sparse fundamental cycles, i.e.
  $$
   \|M\times N\|_{\Lip} = \inf\{\|M\times N\|^A;\; A\in S^{lf,\Lip}_{m+n}(M\times N),\; \text{$A$ is $(m,n)$-sparse}\}.
  $$
 \end{enumerate}
\end{prop}

\begin{proof}
 The second statement is a direct consequence of the first one. To prove the first one it is enough to just apply straightening procedure, but with more carefully chosen sets $(F_j)_{j\in J}$. Choose a family of Borel subsets $(F^M_j)_{j\in J^M}$ of $M$ together with the points $(z^M_j)_{j\in J^M}$ and sections $(s_j^M)_{j\in J}$ with all the properties indicated in the description of the straightening procedure, but with the additional assumption that $\diam(F^M_j)<\frac{E_{m+n,K}}{2}$ and $s^M_j:F_j\rightarrow B_{V_{z^M_j}}(\wt{z_j}^M,\frac{E_{m+n,K}}{2})$ for every $j\in J^M$. Similarly choose a family $(F^N_j)_{j\in J^N}$ of Borel subsets of $N$ together with points $(z_j^{N})_{j\in J^N}$ and sections $(s_j^N)_{j\in J^N}$ and as a base of the straightening procedure for $M\times N$ take a family $(F^M_{j_1}\times F^N_{j_2})_{(j_1,j_2)\in J^M\times J^N}$ together with points $(z^M_{j_1},z^N_{j_2})_{(j_1,j_2)\in J^M\times J^N}$ and sections $(s_{j_1}\times s_{j_2})_{(j_1,j_2)\in J^M\times J^N}$. This family is locally finite, satisfying $\diam(F^M_{j_1}\times F^N_{j_2})<E_{m+n,K}$ and $s_{j_1}\times s_{j_2}:(F_{j_1}\times F_{j_2})\rightarrow V_{(z^M_{j_1},z^N_{j_2})}(E_{m+n,K})$ for every $(j_1,j_2)\in J^M\times J^N$. Hence if $c\in C^{lf,\Lip}_k(M\times N)$ is any locally finite Lipschitz chain it can be straightened with respect to that family. Note also that for any $L<\infty$ and $p\in\mathbb{N}$ the family
 $$
 A_{L,p}:=\{\sigma\in Map(\Delta^{k+l},M\times N);\;\Lip(\sigma)\leq L;\; \sigma \text{ is $p$-piecewise straight simplex}\}
 $$
 belongs to $S^{lf,\Lip}_{k+l}(M\times N)$ and is $(k,l)$-sparse by the construction of $(F^M_{j_1}\times F^N_{j_2})_{(j_1,j_2)\in J^M\times J^N}$ and Lipschitz condition. To finish the proof note that $c\sim \str_p(c)$ for some $p\in\mathbb{N}$, $|c|_1\geq |\str_p(c)|_1$ and $\str_p(c)$ has a support in $A_{L,p}$ for some $L$, thus it is $(k,l)$-sparse.
\end{proof}

\subsection{Proportionality Principle}
Another result obtained in \cite{LS} for non-positively curved manifolds is the proportionality principle for the Lipschitz simplicial volume. We generalize it here and prove Theorem \ref{PP1}. As a corollary we obtain a proof of Theorem \ref{PP0}, based on Thurston's approach \cite{T}, but slightly different from the proof given in \cite{Str}.

The idea of the original proof is as follows. Using the common universal cover one can construct a 'smearing map' from $C^1$-$\ell^1$ locally finite Lipschitz chain complex on $M$ into the chain complex of Borel measures on $C^1(\Delta^*,N)$ with finite variation and Lipschitz determination. This map does not increase the norm and has the property that the image of locally finite real fundamental class of $M$ maps to a (measure) fundamental class of $N$ multiplied by $\frac{\vol(M)}{\vol(N)}$ (or more precisely a measure cycle such that every singular chain homologous to it, if it exists, is an indicated multiplicity of a fundamental cycle). Moreover, the image of this map can be approximated 'isometrically' by a singular locally finite Lipschitz cycle, which finishes the proof.

The usage of $C^1$ chains and measures is strictly technical and is used to recognise the image of the smearing map. In our approach we cannot use $C^1$ chains and measures, however, piecewise $C^1$ chains have all required properties. The following propositions are either taken from~\cite{LS} or are slight modifications of those with the same proofs. For a Riemannian $n$-manifold we will denote by $\dvol_M\in \Omega^n(M)$ the volume form on $M$. Recall that we are able to integrate over $\dvol_M$ not only Lipschitz chains (via Rademacher's theorem), but also Borel measures on $C^1(\Delta^n,M)$ and piecewise $C^1$ measures on $\mathscr{P}C^1(\Delta^n,M)$ via the formula
$$
\int_M \mu\, \dvol_M := \int_{\mathscr{P}C^1(\Delta^n,M)}\int_{\Delta^n} \sigma^*\,\dvol_M\, d\mu(\sigma) 
$$
for $\mu\in\mathcal{C}^{PC^1,\Lip}_*(M)$. We will denote the evaluation of $\dvol_M$ on simplex, chain or measure by $\langle \dvol_M,\cdot\rangle$.

\begin{prop}[{\cite[Proposition 4.4]{LS}}]\label{recognise_class}
 Let $M$ be a Riemannian $n$-manifold, and let $c=\sum_{k\in\mathbb{N}}a_k\sigma_k\in C^{lf}_n(M)$ be a cycle with $|c|_1<\infty$ and $\Lip(c)<\infty$.
 \begin{enumerate}
  \item Then $\langle \dvol_M,\sigma_k\rangle\leq \Lip(c)^n\vol(\Delta^n)$ for every $k\in\mathbb{N}$
  \item Furthermore, we have the following equivalence:
  $$
  \sum_{k\in\mathbb{N}}a_k\cdot\langle \dvol_M,\sigma_k\rangle = \vol(M) \;\Leftrightarrow\;\text{$c$ is a fundamental cycle}.
  $$
 \end{enumerate}
\end{prop}

\begin{remark}\label{fundamental_class}
 The second statement of the above proposition gives an easy criterion to distinguish fundamental class. It can be also applied to a given measure cycle, but only if it is homologous to some singular cycle. The reason is that there is no obvious map $\mathcal{C}_*^{PC^1,\Lip}(M)\rightarrow C^{lf}_*(M)$. However, by Proposition~\ref{measure_hom_theories} we obtain a map $\mathcal{H}^{PC^1,\Lip}_*(M)\rightarrow H^{lf}_*(M)$ by composing the inverse of the isometric isomorphism $H^{PC^1,\ell^1,\Lip}_*(M)\rightarrow\mathcal{H}^{PC^1,\Lip}_*(M)$ and a map $H^{PC^1,\ell^1,\Lip}_*(M)\rightarrow H^{lf}_*(M)$ induced by the inclusion of chains. We can therefore define fundamental cycles in $\mathcal{C}_n^{PC^1,\Lip}(M)$ as the cycles representing any class in the preimage of the fundamental class in $H_n^{lf}(M)$.
\end{remark}

The following proposition is a variation of the results from \cite{LS} and the proof is completely analogous. Let $U$ be a common universal cover of $M$ and $N$ with covering maps $p_M$ and $p_N$ respectively, let $G=Isom^+(U)$ and let $\Lambda=\pi_1(N)$. Then $\Lambda$ is a lattice in $G$ \cite[Lemma 4.2]{LS}. Denote by $\mu_{\Lambda\backslash G}$ the normalized Haar measure on $\Lambda\backslash G$.

\begin{prop}[{\cite[Proposition 4.9, Lemma 4.10]{LS}}]\label{smear}
 Let $\sigma:\Delta^*\rightarrow M$ be a piecewise $C^1$ simplex, and let $\wt{\sigma}:\Delta^*\rightarrow U$ be a lift of $\sigma$ to $U$. Then push-forward of $\mu_{\Lambda\backslash G}$ under the map
 $$
 \smear_{\wt{\sigma}}:\Lambda\backslash G \rightarrow \mathscr{P}C^1(\Delta^*,N),\;\Lambda g\mapsto p_N\circ g\wt{\sigma}
 $$
 does not depend on the choice of the lift of $\sigma$ as is denoted by $\mu_\sigma$. Further there is a well-defined chain map
 $$
 \smear_*:C_*^{PC^1,\ell^1,\Lip}(M) \rightarrow \mathcal{C}_*^{PC^1,\Lip}(N),\; \sum_{\sigma}a_\sigma \sigma\mapsto \sum_{\sigma}a_\sigma \mu_\sigma.
 $$
 Moreover, for every fundamental cycle $c\in C_n^{PC^1,\ell^1,\Lip}(M)$ we have
 $$
 \langle \dvol_N,\smear_n(c) \rangle = \int_{\mathscr{P}C^1(\Delta^n,N)}\int_{\Delta^n} \sigma^*\,\dvol_N\,d\smear_n(c)(\sigma)=\vol(M).
 $$
\end{prop}

\begin{proof}[Proof of theorem \ref{PP1} and \ref{PP0}]
 We will show that
 $$
 \frac{\|N\|_{\Lip}}{\vol(N)}\leq \frac{\|M\|_{\Lip}}{\vol(M)}
 $$
 and the opposite inequality will follow by symmetry. For $\|M\|_{\Lip}=\infty$ the inequality is obvious, so we can assume $\|M\|_{\Lip}<\infty$. By Proposition~\ref{homology_theories} in this case there exists a fundamental cycle in $C^{PC^1,\ell^1,\Lip}_n(M)$. Let $c = \sum_{\sigma}a_\sigma\sigma\in C^{PC^1,\ell^1,\Lip}_n(M)$ be a fundamental cycle and consider its image under the smearing map. It follows from Propositions~\ref{smear}, \ref{recognise_class} and Remark~\ref{fundamental_class} that its image $\smear_n(c)$ represents a fundamental class in $\mathcal{C}_*^{PC^1,\Lip}(N)$ multiplied by $\frac{\vol(M)}{\vol(N)}$. Moreover, by the construction of the smearing map
 $$
 |\smear_n(c)| = |\sum_{\sigma}a_\sigma\mu_\sigma| \leq \sum_\sigma |a_\sigma||\mu_\sigma| = \sum_\sigma |a_\sigma| = |c|_1.
 $$
 By Proposition \ref{measure_hom_theories} there exists a cycle in $C^{PC^1\ell^1,\Lip}_n(N)$ which represents the same homology class as $\smear_n(c)$ with not greater $\ell^1$ norm. Because Proposition \ref{homology_theories} implies that the Lipschitz simplicial volume of $M$ can be computed on piecewise $C^1$ cycles we obtain
 $$
 \|N\|_{\Lip}\leq \frac{\vol(N)}{\vol(M)}\|M\|_{\Lip}\;\Rightarrow\; \frac{\|N\|_{\Lip}}{\vol(N)}\leq \frac{\|M\|_{\Lip}}{\vol(M)}.
 $$
 To prove Theorem \ref{PP0} we need only to use the fact that for closed manifolds the classical and Lipschitz simplicial volumes coincide \cite[Remark 1.4]{LS}.
\end{proof}

\bibliography{references}

\begin{thebibliography}{99}
  \bibitem{B} A. Borel \emph{Compact Clifford-Klein forms of symmetric spaces}, Topology 2 (1963), 111-122
  \bibitem{CE} J. Cheeger, D. Ebin \emph{Comparison theorems in Riemannian geometry}, North-Holland Publishing Company (1975)
  \bibitem{CF} C. Conell, B. Farb \emph{The degree theorem in higher rank}, J. Differential Geom. 65 (2003), no. 1, 19-59
  \bibitem{G} M. Gromov \emph{Volume and bounded cohomology}, Institut des Hautes \'Etudes Scientifiques Publications Math\'ematiques no. 56 (1982), 5-99
  \bibitem{ES} R. Engelking, K. Sieklucki \emph{Geometria i Topologia. Cz\k{e}\'s\'c II: Topologia}, Biblioteka Matematyczna, PWN (1980)
  \bibitem{HAT} A. Hatcher \emph{Algebraic Topology}, Cambridge University Press (2002)
  \bibitem{LafS} J.-F. Lafont, B. Schmidt \emph{Simplicial volume of closed locally symmetric spaces of non-compact type}, Acta Math. 197 (2006), no. 1, 129-143
  \bibitem{Lee} J. M. Lee \emph{Introduction to Topological Manifolds}, Graduate Texts in Mathematics, Springer (2011)
  \bibitem{LoM} C. L\"{o}h \emph{Measure homology and singular homology are isometrically isomorphic}, Math. Z. 253 (2006), no. 1, 197-218
  \bibitem{LS} C. L\"{o}h, R. Sauer \emph{Degree theorems and Lipschitz simplicial volume for non-positively curved manifolds of finite volume}, J. Topol. 2 (2009), 193-225
  \bibitem{Str} C. Strohm (=L\"{o}h) \emph{The Proportionality Principle of Simplicial Volume}, diploma thesis, WWU M\"{u}nster, arXiv:math.AT/0504106 (2004)
  \bibitem{T} W. P. Thurston \emph{The Geometry and Topology of Three-Manifolds}, \url{http://www.msri.org/publications/books/gtm3} (1978)
\end{thebibliography}

\end{document}